\documentclass{article}
\usepackage{maa-monthly}
\usepackage{graphicx}
\usepackage{amssymb,amscd, epsf}
\usepackage{array}

\theoremstyle{theorem}
\newtheorem{theorem}{Theorem}

\theoremstyle{definition}
\newtheorem*{definition}{Definition}
\newtheorem*{remark}{Remark}
\newtheorem{lemma}{Lemma}
\begin{document}

\title{Cubes and Boxes have Rupert's passages in every direction\footnote{This article has been accepted for publication in American Mathematical Monthly, published by Taylor \& Francis.}}
\markright{Rupert's passages}
\author{Andr\'{a}s Bezdek\footnote{Supported by NKFIH grant KKP-133864.} \and
        Zhenyue Guan \and  Mih\'aly Hujter \and Antal Jo\'os }

\maketitle

\begin{abstract}
It is a $300$ year old counterintuitive observation of Prince Rupert of Rhine that in cube a straight tunnel can be cut, through which a second congruent cube can be passed. Hundred years later P. Nieuwland generalized Rupert's problem and asked for the largest aspect ratio so that a larger homothetic copy of the same body can be passed. We show that cubes and in fact all rectangular boxes have Rupert's passages in every direction, which is not parallel to the  faces.  In case of the cube it was assumed without proof that the solution of the Nieuwland's problem is a tunnel perpendicular to the largest square contained by the cube. We prove that this unwarranted assumption is correct not only for  the cube, but also for all other rectangular boxes.
\end{abstract}


\section{Rupert's passages and Nieuwland's constants}

It is commonly agreed that Prince Rupert of the Rhine (1619-1682) was the first person, who posed the following  puzzle.:

\medskip

\noindent Rupert's passage problem: {\it  Cut a hole in a cube, through which another cube of the same size shall be able to pass.}

\begin{figure}[h]
\begin{center}
\includegraphics[scale=.6]{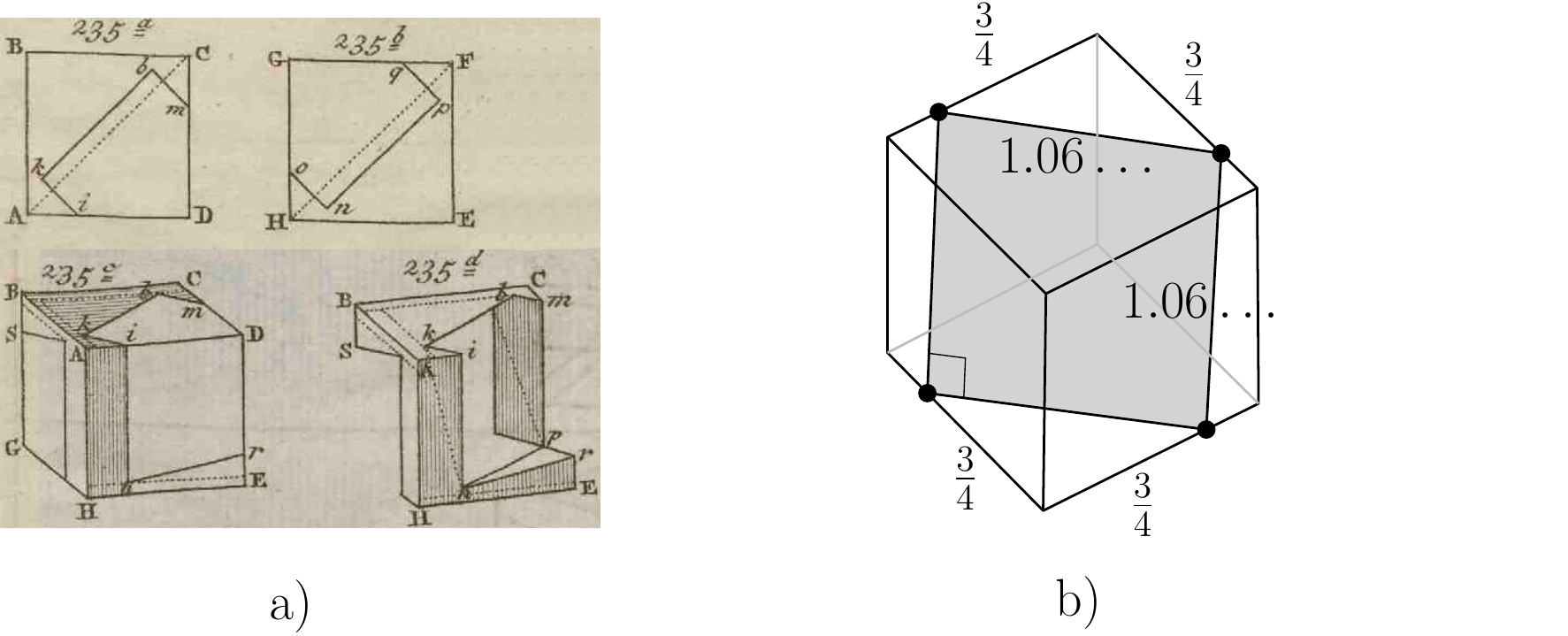}
\caption{ The cube is Rupert. a) is a visual proof from 1816 \cite{swinden16}; b) is a proof without words from today.}
\label{Ruperttunnel}
\end{center}
\end{figure}

The affirmative answer of Rupert's passage problem is attributed to P. Nieuwland and was published in a paper of J.H. van Swinden \cite{swinden16} in 1816. It is common to say that {\it the cube is Rupert}  (see the survey papers \cite{rickey05} and \cite{jerrard04}).  On Figure~\ref{Ruperttunnel}b the four points marked by dots partition their edges in the ratio  $1 : 3$. Repeated use of Pythagorean theorem reveals that these four points form a square of edge length $\frac{3\sqrt 2}4 = 1.06 \dots$ allowing to cut a suitable perpendicular passage \cite{gardner01}.

Rupert's passage problem sounds paradoxial at first. After realizing that the problem is about finding two shadows so that one shadow fits inside the other, the problem becomes manageable and also triggers generalizations.  Which solids are Rupert? If a solid is Rupert, what is the maximum size of the second body so that it still can be moved through the first one?  The second question was asked  for cubes  by P. Nieuwland (1764-1794) almost a century after Rupert's question was posed. The later problem is referred to as {\it Nieuwland's passage problem} and the homothetic constant corresponding to the largest size (or rather the supremum of the possible sizes) is called {\it Nieuwland's constant}. Several survey papers popularized generalizations over the years, including papers of J.H. van Swinden \cite{swinden16} in 1816, of A. Ehrenfeucht \cite{ehrenfeucht64} in 1964, of D.J.E. Schrek \cite{schrek50} in 1950, of  M. Gardner \cite{gardner01} in 2001 and of  V.F. Rickey \cite{rickey05} in 2005.

Over the years various solids were studied concerning Rupert's passage problem. Spheres and bodies of constant widths are obviously not Rupert. Positive answers were found for tetrahedron and for octahedron by Scriba \cite{scriba68} in 1968, for rectangular boxes by Jerrard and Wetzl \cite{jerrard04} in 2004, for universal stoppers by Jerrard and Wetzl in \cite{jerrard08} in 2008, for dodecahedron and for icosahedron by Jerrard, Wetzl and Yuan \cite{jerrard17} in 2017, for eight Archimedean solids by Chai, Yuan and Zamfirescu \cite{chai18} in 2018 and for $n$-cubes by Huber, Schultz and Wetzl \cite{huber18} in 2018.  The most intriguing open question in this area is wether every convex polyhedron is Rupert (see \cite{jerrard17}).

\section{New results concerning  Rupert's passages}

Neither Rupert nor Nieuwland, were precise about how they want to pass a cube through the other one. According to standards at their time, using words like 'passage' and 'tunnel' was enough to indicate that motion is expected to  be a translation. There is another assumption which requires rigorous proof. We noticed this detail in a paper of Jerrard and Wetzel who wrote several interesting papers on Rupert's passages.    \cite{jerrard04} appeared  in this Monthly, and is about the solution of  Nieuwland's type problem for passing  rectangular boxes through a unit cube. In the introduction  the authors state a comonly used unwarranted assumption

\noindent
\begin{quote}
  {\it Nieuwland's passage problem of finding the largest cube that can pass through a unit cube is equivalent to finding the largest square that fits in the unit cube, because once the largest square is located, the hole through the cube having that largest square as its cross section clearly provides the desired passage.}
\end{quote}

\noindent Indeed, once the largest homothetic square  is located, the hole raised perpendicularly over this square clearly provides the desired passage. The other direction of the equivalency statement can not be assumed without proof.  It is very unlikely that Rupert and Nieuwland wanted to restrict passages to tunnels raised perpendicularly over squares contained in the unit cube. Thus, before this equivalency issue is settled the constant found by Nieuwland or later by other authors are only lower bounds for the Nieuwland's constants of the cube and of other solids.

The main goals of this paper were i) to show that cubes are Rupert in every direction and ii) to show equivalency  of the later two problems for cubes. Along the way we noticed that essentially the same arguments, with some modifications, prove the analogous theorems for all rectangular boxes. We will prove

\begin{theorem} \label{rectangle-in-projection}
The interior of every hexagonal projection of a rectangular box with sides $a \leq b \leq c$ contains a rectangle with sides $a$ and $b$. Equivalently, rectangular boxes have Rupert's passages in every direction not parallel to faces. In particular, cubes have Rupert's passages in every direction not parallel to faces.
\end{theorem}

Using Theorem~\ref{rectangle-in-projection} we also prove

\begin{theorem} \label{largestsquare}  Let $B$ be a rectangular box with edge lengths $a \leq b \leq c$. If a homothetic box $\lambda B$ ($\lambda >0$) can be passed  through  $B$ by translation, then the interior of box $B$ contains rectangle with sides $\lambda a$ and $\lambda b$. With other words, the problem of finding Nieuwland constant of a given rectangular box $B$, is equivalent to finding the supremum of $\lambda$'s for which the interior of box $B$ contains a $\lambda$-homothetic copy of the smallest face of box $B$.
\end{theorem}

\section{Preliminary Lemmas}

In this section we prove three lemmas:

\begin{lemma} \label{cube-basics} Let $C$ be a unit cube in a general position in the $xyz$-coordinate space. Denote by $H$ the perpendicular projection of $C$ on the $xy$-plane (Figure~\ref{basic}). Then,
\begin{enumerate}
  \item the projection $H$ is either a central symmetrical hexagon with obtuse angles or a rectangle,
  \item the area of $H$ is equal to the length of the perpendicular projection of $C$ on the $z$-axis,
  \item $p^2+q^2 +r^2=1$, where $p,q$ and $r$ are the third coordinates of the three non-parallel normal vectors of the faces of $C$. Similar equations hold for the second and for the first coordinates.
\end{enumerate}
\end{lemma}

\begin{proof} [Proof of Lemma \ref{cube-basics}] Statement 1  follows from the visual geometric observation that in an infinite wedge bounded by two half planes with an acute angle, no solid cube can reach the edge of the wedge on its convex side. Indeed, if this  would be possible, then a continuity argument would imply that the contact is possible also with one face of the cube  on one of the bounding half planes, a contradiction.

\begin{figure}[h]
\begin{center}
\includegraphics[scale=.5]{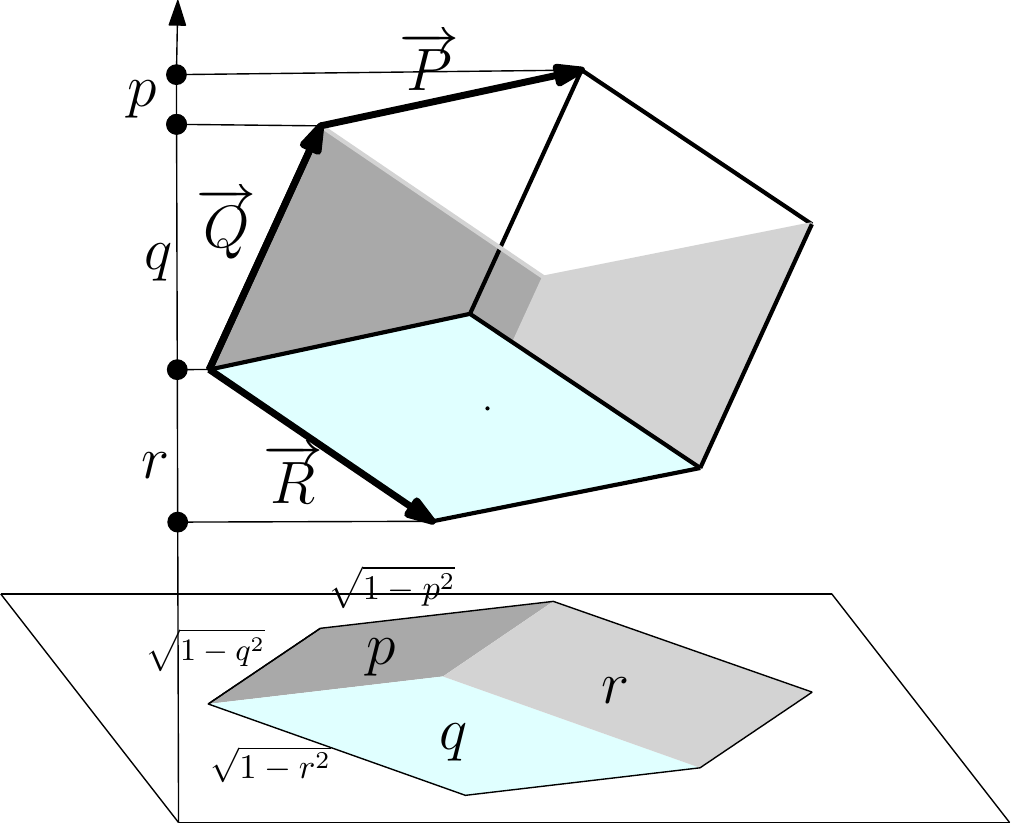}
\caption{Basic properties of the projections of a cube.}
\label{basic}
\end{center}
\end{figure}

Statement 2 is a well known property of the cube.  Figure~\ref{basic} becomes a proof without words, once it is noticed that the dot product of a unit edge vector and the vertical unit vector has two geometric meanings. On the one hand, it is equal to the area of the shadow of that face which is perpendicular to the unit edge vector. On the other hand, it is equal to the length of the perpendicular projection of the unit edge vector on the $z$-axis.

Statement 3 must be also an old known property of the cube.  Here we present a short proof via matrices. Let $\overrightarrow{P} , \overrightarrow{Q}$ and  $\overrightarrow{R}$ be the three normal vectors with third coordinates p, q and r. Let $A$ be the matrix whose row vectors are these normal vectors. The linear transformation $A$ is an isometry so for all column vectors $\alpha , \beta \in \mathcal R^3$, $(A \alpha )^T \cdot (A\beta ) = \alpha^T \cdot \beta$. Equivalently, $(\alpha^T A^T) \cdot (A\beta ) = \alpha^T \cdot \beta$. Thus,  $A^T A =I $, which in view of the row-column multiplication means that the column vectors of $A$ form an orthonormal basis. In particular, the third column vector is a unit vector, thus $p^2+q^2 +r^2=1$.
\end{proof}

Next, we will show Lemma \ref{squareatcorner}, which not only claims that the interior of every hexagonal shadow of a unit cube contains a unit square, but also explains, where such squares are located within a shadow.

\begin{lemma} \label{squareatcorner}
Let $H$ be a hexagonal projection of a unit cube (Figure~\ref{projections2}). Then, there are at least two pairs of opposite vertices of the hexagon $H$ so that each of these four (or six) vertices can share a vertex of a unit square in $B$. Moreover, each of these four (or six) squares are unique and have the following properties:
\begin{itemize}
  \item  The two vertices adjacent to the 'corner' vertex, lie on a pair of opposite sides of $H$.
  \item  The fourth vertices of these squares belong to interior of $H$.
  \item  Each of these squares can be moved in the interior of the hexagon $H$.
\end{itemize}
\end{lemma}

\begin{remark}
First we used the GeoGebra software to see how does the shadow of a cube change when a cube is rotated in $3$-space. The experiment made us  believe that there might exists a unit square in every shadow, so that the sqaure shares a vertex with the projection of the cube. The conjectured special position allowed us to get the affirmative answer by carrying out a straightforward four page long computation, using two variables only. The simple proof bellow is different, it is adjusted to the box version and it reveals more geometric reasons.
\end{remark}

\begin{figure}[h]
\begin{center}
\includegraphics[scale=.5]{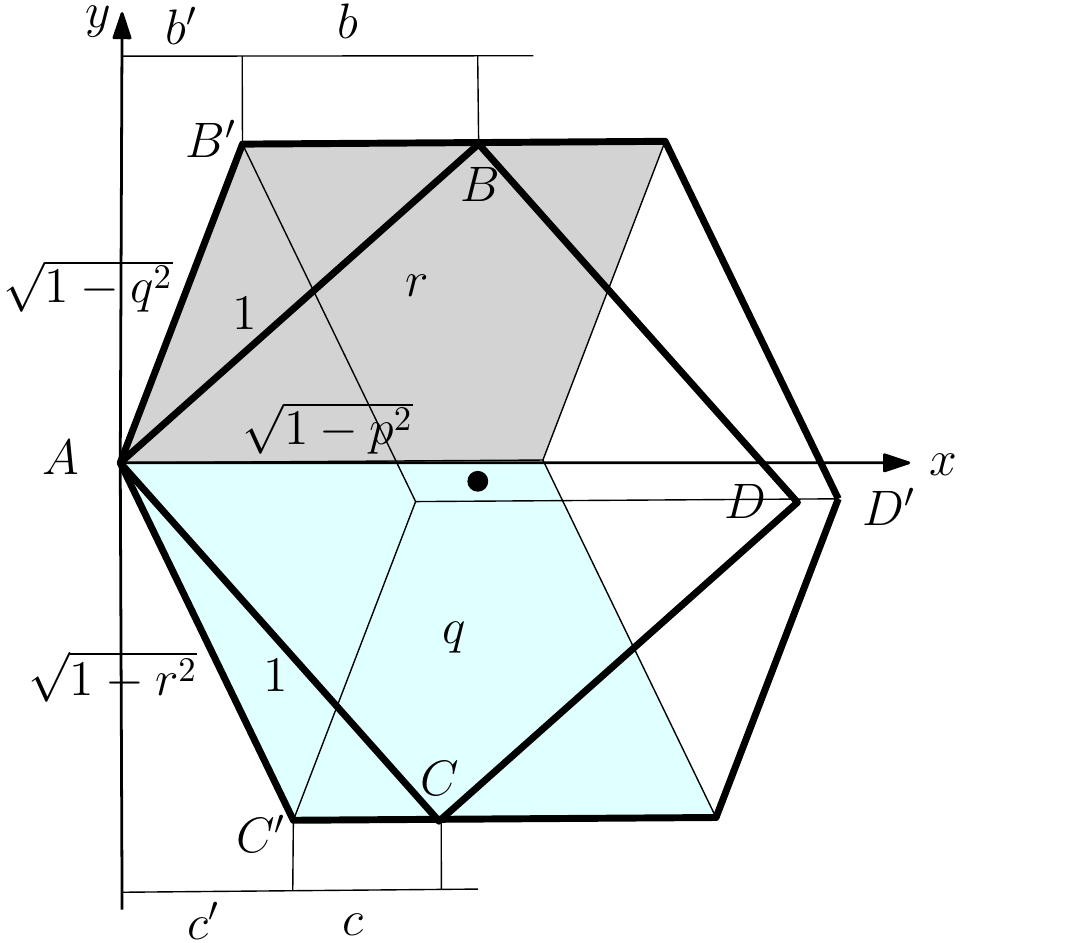}
\caption{Every projection of a unit cube contains a unit square.}
\label{projections2}
\end{center}
\end{figure}

\begin{proof}[Proof of Lemma \ref{squareatcorner}] Let $A$ be one of the  vertices of the hexagonal projection $H$ of the cube. Let $V$ be that vertex of the cube whose vertical projection is $A$. Let us label the three unit edge vectors emanating from $V$ by $\overrightarrow{P}, \overrightarrow{Q}$ and $\overrightarrow{R}$. If $p,q$ and $r$ denote the absolute values of the $z$-coordinates of these vectors, then the lengths of the projections of these vectors are $\sqrt{1-p^2}, \sqrt{1-q^2}$ and $\sqrt{1-r^2}$. Let us orient the hexagon $H$ so that one pair of opposite sides is horizontal and the lengths of these sides equal to $\sqrt{1-p^2}$.  Since all angles of $H$ are obtuse (Part 1 of Lemma~\ref{cube-basics}) we may assume the labelling of side lengths and the areas of the sub-parallelograms are as shown on Figure~\ref{projections2}.

In order to do a little computation, we introduce an $xy$-coordinate system centered at $A$ so that the horizontal line is the $x$ coordinate axis. Let $B'$ and $C'$ be the adjacent vertices of $A$ in $H$. Choose points $B$ and $C$ on the horizontal sides of $H$ at distance $1$ from $A$. Let $B$ the point at distance $b$ from $B'$ and $B'$ at distance $b'$ from the $y$-axis. Distances $c$ and $c'$ are introduced analogously for points $C$ and $C'$. We have

\noindent $(b+b')^2 = 1^2-(\frac r {\sqrt{1- p^2}})^2 = 1- \frac {r^2}{1-p^2} = \frac {q^2}{1-p^2}$,

\noindent $(b')^2 = 1- q^2 - (\frac r {\sqrt{1- p^2}})^2  = \frac {q^2}{1-p^2} - q^2 = \frac {p^2 q^2}{1-p^2}$,

\noindent $b = \frac q {\sqrt{1-p^2}} (1-p) = q \sqrt \frac {1-p}{1+p}$.

\noindent Similar expression holds for $c$: $c =  r \sqrt {\frac {1-p}{1+p}}$.

\noindent Thus, we have the following three equivalent inequalities,

\noindent $b+c  < \sqrt {(1-p^2)} \Leftrightarrow (q+r) \sqrt {\frac {1-p}{1+p}} < \sqrt{1-p^2} \Leftrightarrow  q+r < 1 +p$.

Let $D$ be the vertex which completes $A,B,C$ to a square. Since both $H$ and $ABCD$ are centrally symmetric, $D$ lies on the horizontal line through the vertex of $H$ opposite to $A$. Now, the first and the last equivalent inequalities mean that $D$ is inside of $H$ if $q+r \leq 1 +p$. Since all $p, q, r \in (0,1)$, this holds, if $p$ is the largest or the second largest among $p,q,r$, which proves the first two properties listed in Lemma~\ref{squareatcorner}. A small shift followed by a small rotation of these unit square ensures that there is a unit square which lies in the interior of the hexagon.
\end{proof}

\begin{definition} An infinite vertical cylinder, which is raised  over a horizontal rectangle  will be called {\it rectangular  tube}.
\end{definition}

We will need

\begin{lemma} \label{squaretube}
Every planar cross section of a rectangular tube contains a rectangle congruent to its base.
\end{lemma}

\begin{remark}This lemma is a special case of a very strong theorem of K\'os and T\"or\H ocsik \cite{koos90}. They proved that every convex disc covers its shadow. Here we present a new short proof for Lemma~\ref{squaretube}, so that this paper remains self contained.
\end{remark}

\begin{figure}[h]
\begin{center}
\includegraphics[scale=.8]{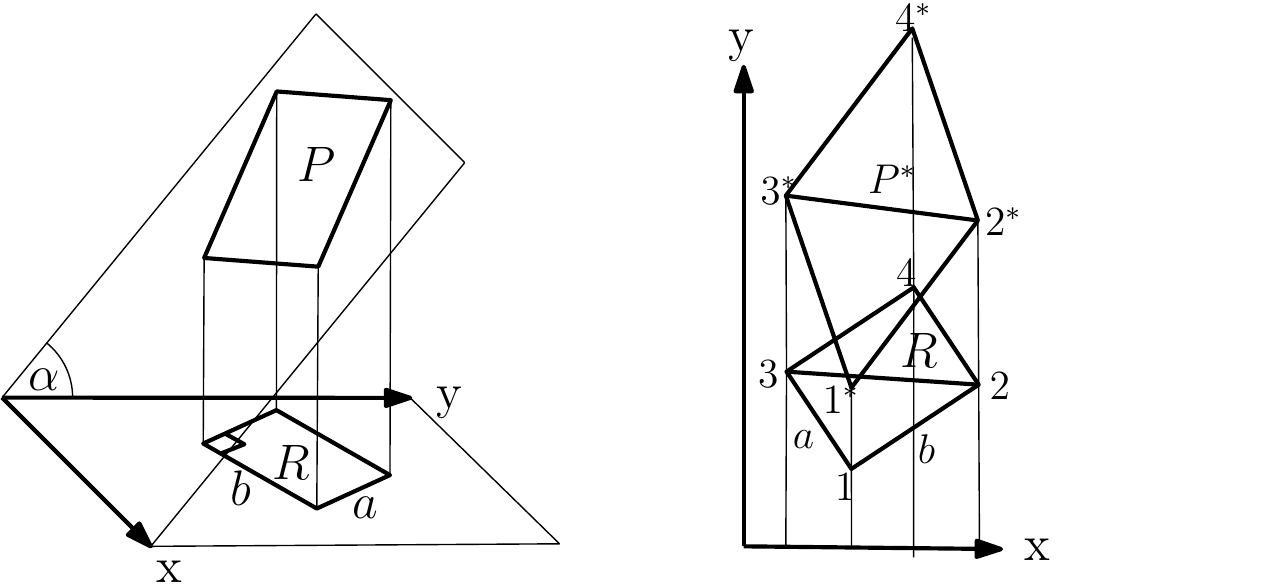}
\caption{The plane of cross section folded over the horizontal plane.}
\label{xy}
\end{center}
\end{figure}

\begin{proof}[Proof of Lemma \ref{squaretube}]
Let $R$ be a rectangle of sides $a$ and $b$, $a\leq b$. If the plane of the cross section is horizontal, then Lemma~\ref{squaretube} is obviously true. Otherwise, denote the angle between the plane of the tube's base $S$ and the plane of the cross section $P$ by $\alpha$ (Figure~\ref{xy}). $P$ is a parallelogram  with opposite sides at least at distances  $a$ and $b$ respectively. Introduce an $x,y$ coordinate system in the horizontal plane of  $R$ so that the $x$ axis belongs also to the plane of $P$.  Let us fold the plane of the cross section into the horizontal plane around the  $x$-axis. It is clear that there is a congruent image $P^*$ of  $P$ so that the linear map $(x,y) \rightarrow (x, \frac 1{\cos \alpha}  y)$ takes the rectangle $R$ to parallelogram $P^*$. We denote the image of a point $p$ under this mapping by $p^*$.

Let us label the vertices of rectangle $R$ by $1,2,3$ and $4$ in the order of their  $y$ coordinates (Figure 1). It is easy to see that $\frac 1{\cos \alpha} > 1$ implies that $P^*$ has obtuse angles at vertices $2^*$ and $3^*$. Depending on how the diagonal $2^* 3^*$  splits the obtuse angle at $2^*$, we distinguish two cases:

\noindent Case 1. Both sub-angles at $2^*$ are greater than equal to the corresponding sub-angles at $2$. In this case a centrally positioned copy of $R$ inside of $P^*$, so that one of the diagonal of $R$ lies on $2^* 3^*$, is contained in the parallelogram $P^*$.

\noindent Case 2. One of the sub-angles, say $\angle 3^* 2^* 4^*$, is  less than  the corresponding sub-angle at $2$. In this case, the length of  the side $2^* 4^*$ is greater than  the length of side $24$, which is $a$. Since the distance between the sides $2^* 4^*$ and $1^* 3^*$ is at least $b$, the rectangle, which has a vertex at $2^*$ and has a side of length $a$ on the segment  $2^* 4^*$ is contained in the parallelogram $P^*$.
\end{proof}

\section{Proof of Theorems}

\begin{proof} [Proof of Theorem~\ref{rectangle-in-projection}] Let $B$ be a rectangular box with edge lengths $a \leq b \leq c$. Shorten the longest edges of $B$ to $b$. It is enough to show that the new box $B'$ contains a rectangle of dimensions $a,b$. Let hexagon $H$ be a shadow of $B'$. Two of the six sides of $H$ are projections of edges of length $a$, while the remaining four sides are projections of  edges of length $b$. To depict this property we label the sides and their lengthes with $pr(a)$ or $pr(b)$  on Figure~\ref{projections}. We have two pairs of opposite vertices which are endpoints of $pr(a)$  sides. In view of Lemma~\ref{squareatcorner} at least one of these two pairs have vertices, say $A$ and $A'$,  so that moving a cube with edge length $a$ to the corner of the pre-image  of $A$, the projection of the cube contain a square whose vertex is $A$. Now it is easy to see that one can extend one pair of parallel sides of this square to $b$ so that the new rectangle is still contained in $H$, what we wanted.  A small shift followed by a small rotation of this rectangle ensures that there is a unit square which lies in the interior of the hexagon.
\end{proof}

\begin{figure}[h]
\begin{center}
\includegraphics[scale=.6]{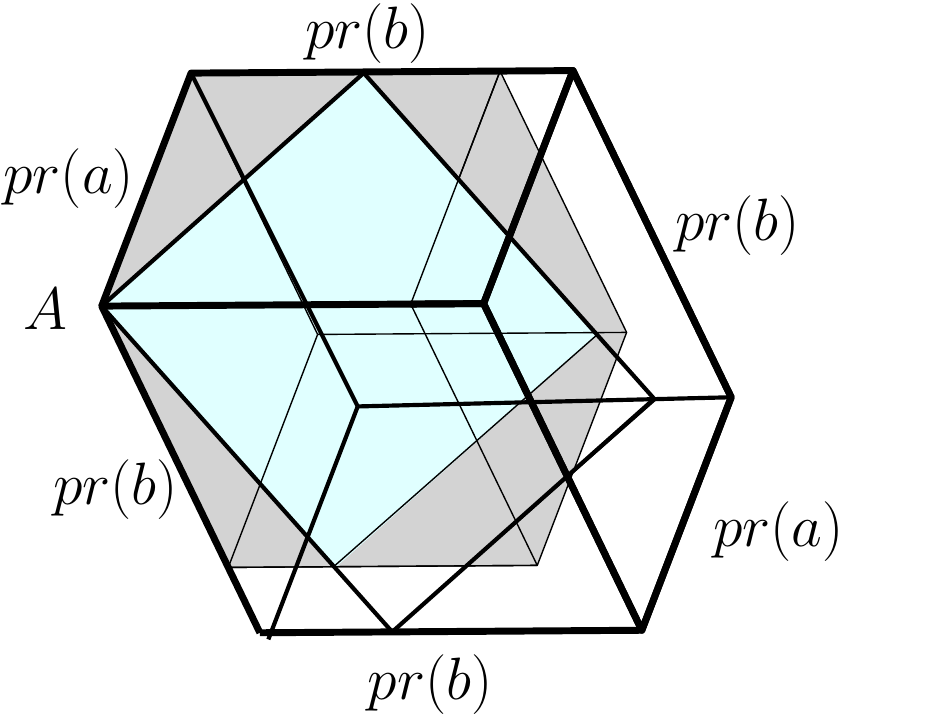}
\caption{Every projection of a box contains a copy of the smallest face. }
\label{projections}
\end{center}
\end{figure}

\begin{proof} [Proof of Theorem~\ref{largestsquare}] Assume that a homothetic box $\lambda B$ can be passed through a stationary box $B$ with sides $a \leq b \leq c$. Although  Theorem~\ref{largestsquare} is obviously true if $\lambda \leq 1$, we present an argument which works for all $\lambda$. On Figure~\ref{largest}, which illustrates the proof, we choose a $\lambda$ much smaller than $1$ so that the figure becomes less crowded and the proof can be better understood. Let $H_1$ be the shadow of the stationary rectangular box $B$ and $H_{\lambda}$ be the shadow of the traveling homothetic rectangular box $\lambda B$. We proved already that the hexagon $H_{\lambda}$ contains a rectangle $\lambda  R$ of sides  $\lambda a$ and $\lambda b$. Since both $H_1$ and $\lambda  R$ are central symmetric, the rectangle $\lambda  R$ can be shifted in center position while remaining inside of the hexagon $H_1$. Now, choose two adjacent vertices, say $U$ and $V$, of this relocated rectangle. Since $U, V$ are inside of the shadow of $B$, we can choose two points  $U^*$, $V^*$ inside of the box $B$ directly over $U$ and $V$. Finally, reflect $U^*$, $V^*$ through the center of the box $B$ two get a parallelogram $P$ in box $B$. $P$'s shadow is the rectangle of sides $a$ and $b$.  According to Lemma~\ref{squaretube},  parallelogram $P$ contains a rectangle of sides $a$ and $b$.
\end{proof}

\begin{figure}[h]
\begin{center}
\includegraphics[scale=.8]{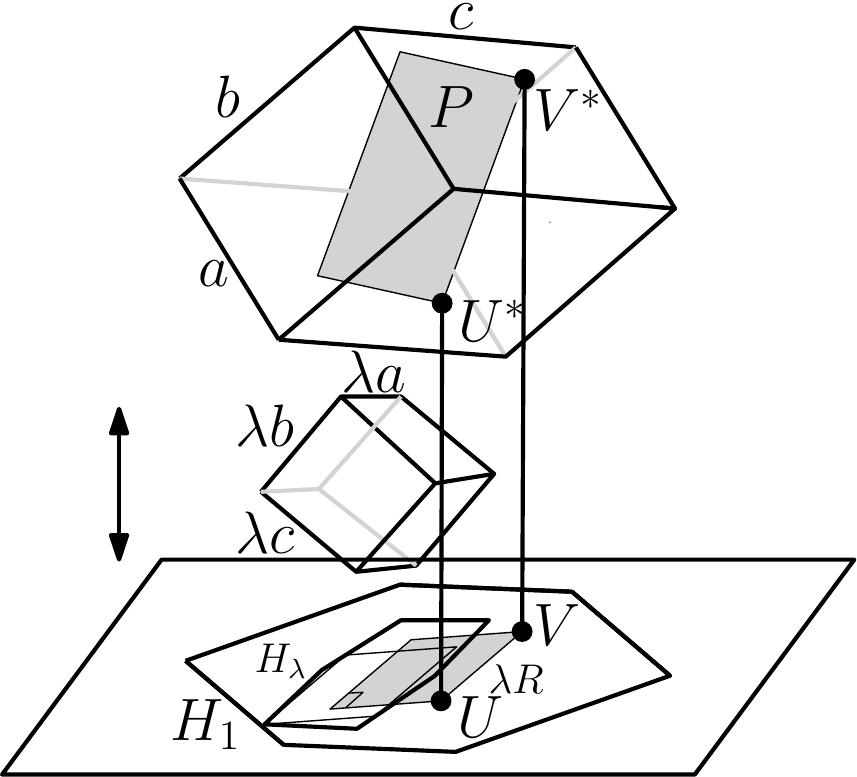}
\caption{Proof of Theorem~\ref{largestsquare}.}
\label{largest}
\end{center}
\end{figure}


\begin{thebibliography}{15}

\bibitem{chai18} Y. Chai, L. Yuan and T. Zamfirescu, \emph{Rupert properties of archimedean solids}, \textit{Amer. Math. Monthly} \textbf{ 125:6} (2018), 497--504.


\bibitem{gardner01} M. Gardner,  \textit{The colossal book of mathematics}, W.W. Norton, New York, 2001.

\bibitem{ehrenfeucht64} A. Ehrenfeucht, \textit{The cube made interesting}, Pergamon Press, Oxford, 1964.

\bibitem{huber18} G. Huber, K.P. Schultz and J.E. Wetzel, The n-cube is Rupert \textit{Amer. Math. Monthly} \textbf{ 125:6} (2018), 505--512.

\bibitem{jerrard08} R.P. Jerrard and J.E. Wetzel, Universal stoppers are Rupert \textit{College Math. J.}  \textbf{ 39:2} (2008), 90--94.

\bibitem{jerrard04} R.P. Jerrard and J.E. Wetzel, Prince Rupert's rectangles \textit{Amer. Math. Monthly} \textbf{ 111} (2004), 22--31.

\bibitem{jerrard17} R.P. Jerrard, J.E. Wetzel and L. Yuan, Platonic  passages \textit{Math. Mag.} \textbf{ 90:2}(2017), 87--98.

\bibitem{koos90} G. K\'os and J. T\"or\H ocsik,  Convex discs can cover their shadows \textit{Discr. Comp. Geom.}  \textbf{ 5} (1990), 529--531.

\bibitem{rickey05} V.F. Rickey, \textit{D\"urer’s Magic Square, Cardano’s Rings, Prince Rupert’s Cube, and Other Neat Things} MAA Short Course "Recreational Mathematics", Albuquerque, New Mexico, 23 August (2005).

\bibitem{schrek50} D.J.E. Schrek,  Prince Rupert's problem and its extensions by Pieter Nieuwland \textit{Scripta Math.} \textbf{ 16} (1950), 73--80, 261--267.

\bibitem{scriba68} C.J. Scriba,  Das problem des Prinzen Ruprecht von der Pfalz \textit{Praxis der Math.} \textbf{ 10(9)} (1968), 241--246.

\bibitem{swinden16} J.H. Van Swinden,  \textit{Grondbeginsels der Meetkunde}, 2nd ed., Amsterdam, 1816, 512--513.

\bibitem{wetzel00} J.E. Wetzel, Rectangles in rectangles \textit{Math. Mag.}  \textbf{ 73} (2000), 204--211.

\end{thebibliography}
\end{document}